\documentclass[12pt]{amsart}
\usepackage{amsmath,amsfonts,graphics,graphicx}

\newcommand\zo{{$(0,1)$\ }}
\newcommand\pq{{\it{P}-\it{Q}}}
\newcommand\ppinv{$P\mbox{-}P^{-1}$}
\newcommand\imag{{\iota}}
\DeclareMathOperator{\SDA}{SDA}
\newcommand\Adam{{\'Ad\'am}}
\newcommand\Z{{\mathbb Z}}
\newcommand\Q{{\mathbb Q}}
\newcommand\C{{\mathbb C}}
\newcommand\settA{{\mathfrak{A}}}
\newcommand\settB{{\mathfrak{B}}}
\newcommand\setA{{\mathcal{A}}}
\newcommand\setB{{\mathcal{B}}}
\newcommand\tth{^{\operatorname{th}}}

\newtheorem{lemma}{Lemma}[section]
\newtheorem{theorem}[lemma]{Theorem}
\newtheorem{corollary}[lemma]{Corollary}
\newtheorem{proposition}[lemma]{Proposition}

\theoremstyle{definition}
\newtheorem{definition}[lemma]{Definition}

\theoremstyle{remark}
\newtheorem*{remark}{Remark}

\begin{document}

\title[Equivalence of Sparse Circulants]{Equivalence of Sparse Circulants:
 the Bipartite \'Ad\'am Problem}

\author{Doug Wiedemann}
\address{
  Center for Communications Research,
  805 Bunn Drive,
  Princeton, NJ 08540
}
\email{doug@idaccr.org}

\author{Michael~E. Zieve}
\address{
  Hill Center,
  Department of Mathematics,
  Rutgers University,
  110 Frelinghuysen Road,
  Piscataway NJ 08854
}
\email{zieve@math.rutgers.edu}

\begin{abstract}
We consider $n$-by-$n$ circulant matrices having entries $0$ and $1$.
Such matrices can be identified with sets of residues mod~$n$, corresponding
to the columns in which the top row contains an entry $1$.  Let $A$ and $B$
be two such matrices, and suppose that the corresponding residue sets $S_A,S_B$
have size at most $3$.  We prove that the following are equivalent:
(1) there are integers $u,v$ mod~$n$, with $u$ a unit, such that $S_A=uS_B+v$;
(2) there are permutation matrices $P,Q$ such that $A=PBQ$.  Our proof relies
on some new results about vanishing sums of roots of unity.  We give
examples showing this result is not always true for denser circulants, as
well as results showing it continues to hold in some situations.  We also
explain how our problem relates to the \'Ad\'am problem on isomorphisms
of circulant directed graphs.
\end{abstract}

\date{\today}

\maketitle

\section{Introduction}

We define a \emph{circulant} to be any square matrix whose
rows are consecutive right circular shifts of each other.
In other words, it is any $n$-by-$n$ matrix $(a_{i,j})$ where
$a_{i,j}$ depends only on $i-j$ mod $n$.
Thus, a circulant is a special type of Toeplitz matrix.
Circulant matrices occur in numerous applications and have been
studied extensively; for instance, see \cite{Davis}.

Surprisingly, for many applications the interest in circulants does
not directly stem from the circular symmetry just described.
For example, every Desarguesian finite projective plane can be represented
as a circulant, by a theorem of Singer.
In other words, if $m$ is a prime power, there is an $n$-by-$n$ circulant
matrix
(where $n=m^2+m+1$) in which each row has $m+1$ entries being $1$ and the
rest being $0$,
with the further condition that the componentwise product of any two rows has
exactly one $1$.
Here the rows represent lines and the columns points, where a line contains
a point whenever the corresponding entry in the circulant is a $1$.

In this paper we are primarily interested in circulants with entries $0$ and
$1$, which we call \zo circulants.  One can think of a \zo circulant as
an incidence structure, or as the nonzero block of the
adjacency matrix of a bipartite graph.
Many interesting examples of incidence structures correspond to
\zo circulants; for instance, these include the much-studied subject of
``cyclic difference sets''.

The \emph{weight} of a \zo circulant is
the number of $1$'s in each row.
Often the $n$-by-$n$ circulants of interest have weight quite small compared
to $n$.  In this paper we prove that,
for circulants of weight at most $3$, various equivalence relations are
the same.  We need some notation to state our result.  For any $n$-by-$n$
\zo circulant $A$, let $S_A$ be the set of integers mod~$n$ corresponding
to the columns in which the top row of $A$ contains an entry $1$; here the
leftmost column is labeled $0$, the next is $1$, and so on.  Also, $A^T$
denotes the transpose of the matrix $A$.

\begin{theorem}
\label{thm intro}
Let $A$ and $B$ be two $n$-by-$n$ \zo circulants of weight at most $3$.  Then the
following are equivalent:
\begin{enumerate}
\item There exist $u,v\in\Z/n\Z$ such that $\gcd(u,n)=1$ and $S_A=uS_B+v$.
\item There are $n$-by-$n$ permutation matrices $P,Q$ such that $A=PBQ$.
\item There is an $n$-by-$n$ permutation matrix $P$ such that
$AA^T=PBB^TP^{-1}$.
\item The complex matrices $AA^T$ and $BB^T$ are similar.
\end{enumerate}
\end{theorem}

It is not difficult to prove that $(1)\Rightarrow(2)\Rightarrow(3)
\Rightarrow(4)$.
The bulk of our effort in proving Theorem~\ref{thm intro} is devoted to
proving $(4)\Rightarrow(1)$ in the case of weight $3$ (smaller weights are
easier to handle).  We give counterexamples to this result for every
weight larger than $3$.  However, for weights $4$ and $5$, the result can
be salvaged to some extent: we show that it holds as long as every prime
factor of $n$ is sufficiently large.
On the other hand, for each weight exceeding $5$, we give counterexamples
for every sufficiently large $n$.

In the context of isomorphisms of circulant graphs, many authors have
studied questions of a similar flavor as Theorem~\ref{thm intro}.
In that context, we restrict to $v=0$ in condition (1) and to $Q=P^{-1}$
in (2), getting conditions (1') and (2').  The equivalence of (1') and (2')
is known as the \Adam\ problem; it is not always true, but the combined
efforts of several authors have shown precisely when it
holds (cf.\ \cite{AP,babai,muzychuk,palfy} and the references therein).
Likewise,
the equivalence of (1') to the similarity of $A$ and $B$ is called the
spectral \Adam\ problem, and is still being studied \cite{D,ET,HC,MPS}.
The equivalence of
(1) and (2) amounts to a bipartite analogue of the \Adam\ problem.  From
the perspective of circulant matrices, condition (2) is a more natural
equivalence than (2'), and condition (2) arises in various applications.  

In the next section we prove some preliminary results about the various
equivalence relations under consideration.  Then in Section~\ref{sec 2}
we give a quick proof of Theorem~\ref{thm intro} in case the weight $k$ is
at most $2$.  The next three sections prove Theorem~\ref{thm intro} in
the much more difficult case $k=3$: in Section~\ref{sec SDA} we reduce the
problem to a question about vanishing sums of roots of unity, which we
resolve in Sections~\ref{sec roots} and \ref{sec 3}.  We discuss the cases
$k=4$ and $k=5$ in Section~\ref{sec 4}, and the case $k\ge 6$ in
Section~\ref{sec 6}.
In Section~\ref{sec adam} we explain how our problem relates to the \Adam\ 
problem.  Finally, in
Section~\ref{sec ack} we suggest some directions for future research.


\section{Equivalence Classes of Circulants}
\label{sec equiv}

In almost any application of circulants, the first row can be
replaced with any circular shift of itself.
For example, both circulants below represent the projective plane of
order 2 (also known as the Fano plane):

\[
\left(
\begin{matrix}
0 & 1 & 1 & 0 & 1 & 0 & 0 \\
0 & 0 & 1 & 1 & 0 & 1 & 0 \\
0 & 0 & 0 & 1 & 1 & 0 & 1 \\
1 & 0 & 0 & 0 & 1 & 1 & 0 \\
0 & 1 & 0 & 0 & 0 & 1 & 1 \\
1 & 0 & 1 & 0 & 0 & 0 & 1 \\
1 & 1 & 0 & 1 & 0 & 0 & 0
\end{matrix}
\right) \quad\text{and}\quad 
\left(
\begin{matrix}
1 & 1 & 0 & 1 & 0 & 0 & 0 \\
0 & 1 & 1 & 0 & 1 & 0 & 0 \\
0 & 0 & 1 & 1 & 0 & 1 & 0 \\
0 & 0 & 0 & 1 & 1 & 0 & 1 \\
1 & 0 & 0 & 0 & 1 & 1 & 0 \\
0 & 1 & 0 & 0 & 0 & 1 & 1 \\
1 & 0 & 1 & 0 & 0 & 0 & 1
\end{matrix}
\right).
\]

Applying a circular shift to every row of a circulant has the effect of
applying a circular shift to the order of the rows, which is equivalent to
applying a circular shift to the columns.
The top row of a circulant can be identified with a set of residues mod~$n$,
by listing the positions containing an entry of $1$; here the leftmost position
is labeled $0$, the next is labeled $1$, and so on.  Thus, the circulants above
are identified with the sets $\{1,2,4\}$ and $\{0,1,3\}$ mod~$7$, which we will
denote as $\{1,2,4\}_7$ and $\{0,1,3\}_7$.

We now establish some notation that will be used throughout the paper.
Let $S$ be the unit circular shift on $n$-dimensional vectors over $\C$.
Thus, if $e_i$ is a column unit vector with a $1$ in position $i$ and a $0$
elsewhere, then $S$ is defined by $S e_i = e_{i-1}$ for $i=0,1,...,n-1$, where
the indices are computed mod~$n$.  In other words, $S$ is the matrix
\[
\left(
\begin{matrix}
0 & 1 & 0 & 0 & \cdots & 0 \\
0 & 0 & 1 & 0 & \cdots & 0 \\
\vdots & \vdots &&\ddots&&\vdots \\
0&0&\cdots&0&1&0 \\
0&0&\cdots&0&0&1\\
1 & 0 & 0 &\cdots & 0 & 0
\end{matrix}
\right).
\]
Note that the transpose of $S$ equals the inverse of $S$, and also
the powers of $S$ are precisely the circular shifts by various amounts.
The circulant corresponding to the residue set $\{a_1,...,a_k\}_n$ is
\[
A = S^{a_1} + ... + S^{a_k}.
\]

For the applications we have in mind, we want to consider two circulants
equivalent if one can be obtained from the other by row and column
permutations.  This motivates the following definition:

\begin{definition} Two circulants $A$ and $B$ are said to be
\pq\ {\it equivalent}
if there exist permutation matrices $P$ and $Q$ such that $B = PAQ$.
\end{definition}

This clearly defines an equivalence relation.
However, note that for most choices of $n$-by-$n$ matrices $A,P,Q$, where $A$
is a \zo circulant and $P$ and $Q$ are permutation matrices, the matrix $PAQ$
will not be circulant.  In order that $PAQ$ be circulant, $P$ and $Q$ must be
quite special.

Applying a circular shift to a circulant has the effect of adding a constant to
its set of residues,
i.e., applying an element of the additive group of residues mod~$n$.
This operation leads to a \pq\  equivalent matrix, since, we can take $P$
to be the shift and $Q$ to be the identity.
It is also true, but less obvious, that multiplication by a unit $u$ mod $n$
produces a \pq\  equivalent circulant: here we choose $P:e_i\mapsto e_{u  i}$
and $Q:=P^{-1}$,
so if the residue set of $A$ is $\{a_1,...,a_k\}_n$ then the residue set of $PAQ$
is $\{u a_1 ,..., u  a_k \}_n$.

We restate this as the following definition and proposition.

\begin{definition} Two subsets $\setA,\setB\subseteq\Z/n\Z$ are
\emph{linearly equivalent}
if there is a unit $u$ in $\Z/n\Z$ such that $\setB=u\setA$.
The two sets are \emph{affinely equivalent} if there exist $u,v\in\Z/n\Z$,
with $u$ a unit, such that $\setB=u\setA+v$.
\end{definition}

\begin{proposition}
\label{prop affpq}
If two subsets of $\Z/n\Z$ are affinely equivalent,
then the associated \zo  circulants are \pq\  equivalent.
\end{proposition}

The main focus of this paper is on the converse of this result.
The following example shows that the converse is not always true.
It exhibits an explicit \pq\  equivalence between the circulants
having residue sets
\[
\{ 0,1,4,7 \}_8 \quad\text{and}\quad \{ 0,1,3,4 \}_8.\]
These sets are
affinely inequivalent since the first set has two
arithmetic progressions of length $3$ ($1,4,7$ and $7,0,1$) but the
second set has none.

\[
\begin{pmatrix}
1 & 1 & 0 & 0 & 1 & 0 & 0 & 1\\
1 & 1 & 1 & 0 & 0 & 1 & 0 & 0\\
0 & 1 & 1 & 1 & 0 & 0 & 1 & 0\\
0 & 0 & 1 & 1 & 1 & 0 & 0 & 1\\
1 & 0 & 0 & 1 & 1 & 1 & 0 & 0\\
0 & 1 & 0 & 0 & 1 & 1 & 1 & 0\\
0 & 0 & 1 & 0 & 0 & 1 & 1 & 1\\
1 & 0 & 0 & 1 & 0 & 0 & 1 & 1
\end{pmatrix}
=
\begin{pmatrix}
1 & 0 & 0 & 0 & 0 & 0 & 0 & 0\\
0 & 0 & 0 & 0 & 0 & 1 & 0 & 0\\
0 & 0 & 0 & 0 & 0 & 0 & 1 & 0\\
0 & 0 & 0 & 1 & 0 & 0 & 0 & 0\\
0 & 0 & 0 & 0 & 1 & 0 & 0 & 0\\
0 & 1 & 0 & 0 & 0 & 0 & 0 & 0\\
0 & 0 & 1 & 0 & 0 & 0 & 0 & 0\\
0 & 0 & 0 & 0 & 0 & 0 & 0 & 1
\end{pmatrix}
\cdot R
\]
where

\[
R =
\begin{pmatrix}
1 & 1 & 0 & 1 & 1 & 0 & 0 & 0\\
0 & 1 & 1 & 0 & 1 & 1 & 0 & 0\\
0 & 0 & 1 & 1 & 0 & 1 & 1 & 0\\
0 & 0 & 0 & 1 & 1 & 0 & 1 & 1\\
1 & 0 & 0 & 0 & 1 & 1 & 0 & 1\\
1 & 1 & 0 & 0 & 0 & 1 & 1 & 0\\
0 & 1 & 1 & 0 & 0 & 0 & 1 & 1\\
1 & 0 & 1 & 1 & 0 & 0 & 0 & 1
\end{pmatrix}
\cdot
\begin{pmatrix}
1 & 0 & 0 & 0 & 0 & 0 & 0 & 0\\
0 & 1 & 0 & 0 & 0 & 0 & 0 & 0\\
0 & 0 & 0 & 0 & 0 & 0 & 1 & 0\\
0 & 0 & 0 & 0 & 0 & 0 & 0 & 1\\
0 & 0 & 0 & 0 & 1 & 0 & 0 & 0\\
0 & 0 & 0 & 0 & 0 & 1 & 0 & 0\\
0 & 0 & 1 & 0 & 0 & 0 & 0 & 0\\
0 & 0 & 0 & 1 & 0 & 0 & 0 & 0
\end{pmatrix}.
\]

We now derive a crucial property of \pq\ equivalent circulants.
If $A$ and $B$ are \pq\ equivalent circulants, then
\[ BB^T = PAQQ^TA^TP^T = PAA^TP^{-1},\]
since the the transpose of a permutation matrix is its inverse.
This proves a necessary condition for \pq\ equivalence:

\begin{proposition}
\label{prop pqspec}
If the circulants $A$ and $B$ are \pq\  equivalent
then $AA^T$ and $BB^T$ are similar matrices, and in fact there is
a permutation matrix which conjugates one to the other.
\end{proposition}

We call $AA^T$ the {\it autocorrelation matrix} of the circulant $A$.
In the example above, the autocorrelation matrices are actually equal,
not just similar.
Note that if $A$ is a circulant then $AA^T$ is also a circulant,
although if $A$ is a \zo  circulant then $AA^T$ might not be \zo-valued:
if $ A = S^{a_1} + ... + S^{a_k} $ then
$AA^T = \sum_{i=1}^{k} \sum_{j=1}^{k} S^{a_i-a_j}.$

In order to discuss when circulant matrices are similar, we first
compute their eigenvalues.  We do this via the well-known method for
diagonalizing circulant matrices.
Let $\imag$ be the square root of $-1$ which lies in the upper half-plane.
Let $\zeta = e^{ 2 \pi \imag / n } $ and let $V$ be the Vandermonde matrix
$( \zeta^{i j} )_{0\le i,j\le n-1}$.
It turns out that $V$ diagonalizes every $n$-by-$n$ circulant.
Let $A = \sum_{i=0}^{n-1} w_i S^i $ with $w_i\in\C$.
Then $V^{-1} A V = D $ is a diagonal matrix with
$D_{r,r} = \sum_i w_i \zeta^{i r}.$
(One way to prove this is to verify directly that $AV = VD$, and then use
 the invertibility of the Vandermonde matrix $V$.)
As a result, if $\{ a_1,...,a_k \}_n$ is the residue set for a \zo circulant
$A$ then the multiset of eigenvalues of $AA^T$ is
\begin{equation}\label{eigen}
\left\{ \sum_{1 \le i,j \le k} \zeta^{(a_i-a_j)r}\colon 0\le r
\le n-1\right\}.
\end{equation}

\begin{proposition}\label{lcm_prop} If $A$ is the circulant
with residue set $\{a_1,\dots,a_k\}_n$,
then the number of times that $k^2$ occurs as an eigenvalue
of $A A^T$ is equal to $\gcd(n,\{a_i-a_j:1\le i,j\le k\})$.
\end{proposition}

\begin{proof}
If $k=0$, the gcd is $n$ so the statement is true because $A=0$.
Now assume $k>0$.  As above, the eigenvalues of $A A^T$ are in bijection
with the values $r\in\Z/n\Z$, where $r$ corresponds to the sum in
(\ref{eigen}).  This sum has $k^2$ terms of unit
magnitude, so it equals $k^2$ if and only if each term is $1$.
This happens if and only if $rg \equiv 0 \pmod n$, where
$g = \gcd\ ( \{ a_i - a_j:1\le i,j\le k \} ) $;
this can be restated as $r\equiv 0\pmod {n/\gcd(n,g)}$.  Finally, the number of
values $r\in\Z/n\Z$ with this property is $\gcd(n,g)$.
\end{proof}


\section{The case $k\le 2$}
\label{sec 2}

In this section we show that, for $k\le 2$, two $n$-by-$n$ \zo circulants of
weight $k$ are affinely equivalent if and only if they are \pq\ equivalent;
in fact, we show that these properties are equivalent to similarity of the
autocorrelation matrices.  If $k\le 1$ this is clear, since all $n$-by-$n$
\zo circulants of weight $k$ are affinely equivalent.
For $k=2$ the result is contained in the following theorem.

\begin{theorem}
\label{thm 2}
The number of affine classes of $n$-by-$n$ weight-$2$
\zo circulants is $\tau (n) -1$, where $\tau (n)$ denotes the number of
divisors of $n$.
Furthermore, weight-$2$ circulants in distinct affine classes have dissimilar
autocorrelation matrices.
\end{theorem}
\begin{proof}
Let $\{ a_1, a_2 \}_n$ be the residue set corresponding to a weight-$2$
circulant.  Put $g := \gcd(n,a_2-a_1)$.
Plainly $\{ a_1,a_2\}_n$ is affinely equivalent to $\{ 0, a_2-a_1\}_n$,
which is linearly equivalent to $\{0,g\}_n$.  Conversely, by
Proposition~\ref{lcm_prop}, if $g,g'$ are divisors of $n$ with
$1\le g<g'<n$, then $\{0,g\}_n$ and $\{0,g'\}_n$ correspond to circulants
with dissimilar autocorrelation matrices.  The result follows.
\end{proof}


\section{Preliminaries for larger $k$}
\label{sec SDA}

In the previous section we proved Theorem~\ref{thm intro} for $k\le 2$.
The remaining case $k=3$ is vastly more difficult.  In this section we
begin our attack on this case by showing that two weight-$3$ circulants
are affinely equivalent if and only if their autocorrelation matrices
are linearly equivalent.  More explicitly, if
$\setA = \{ a_1, ..., a_k \}$ is a set of $k$ residues mod $n$,
let $\Delta (\setA ) $ be the multiset of $k^2$ residues
$\{ a_i-a_j \mid 1 \le i,j \le k\}$.
If $\setA$ is affinely equivalent to another residue set $\setB$,
then there is a unit $u$ mod~$n$ such that $\Delta(\setA)=u\Delta(\setB)$.
We will prove the converse when $k=3$.
Since $u\Delta(\setB)=\Delta(u\setB)$, it suffices to prove the converse
in case $u=1$.

\begin{definition}
For positive integers $n$ and $k$,
the {\it Same Difference Assertion}, or $\SDA(n,k)$, is the following
assertion: for any sets $\setA,\setB$ of $k$ residues mod $n$, we have
$\Delta(\setA) = \Delta(\setB)$ if and only if $\setA$ is affinely
equivalent to $\setB$.
\end{definition}

\begin{remark}
As noted above, $\SDA(n,k)$ is equivalent to saying that, for any two
$n$-by-$n$ \zo circulants of weight $k$, affine equivalence of the
circulants is equivalent to linear equivalence of the autocorrelation
matrices.
\end{remark}

\begin{proposition}\label{prop SDA3}
$\SDA(n,3)$ is true for each $n>0$.
\end{proposition}

\begin{proof}
Suppose $\setA$ and $\setB$ are order-$3$ subsets of $\Z/n\Z$
with $\Delta(\setA)=\Delta(\setB)$.  Identify elements of $\Z/n\Z$
with points on the circle of circumference $n$ centered at the origin,
via $j\mapsto e^{2\pi \imag j/n} n/(2\pi)$.  In this way we can
identify $\setA$ and $\setB$ with sets of 3 points on this circle,
where the arclength between two points equals the difference between
the corresponding elements of $\Z/n\Z$.
In the following diagram, $\setA$ is recorded in the left-hand circle,
and $\setB$ is recorded in the right-hand circle.

\par\noindent
\scalebox{0.70}{\includegraphics{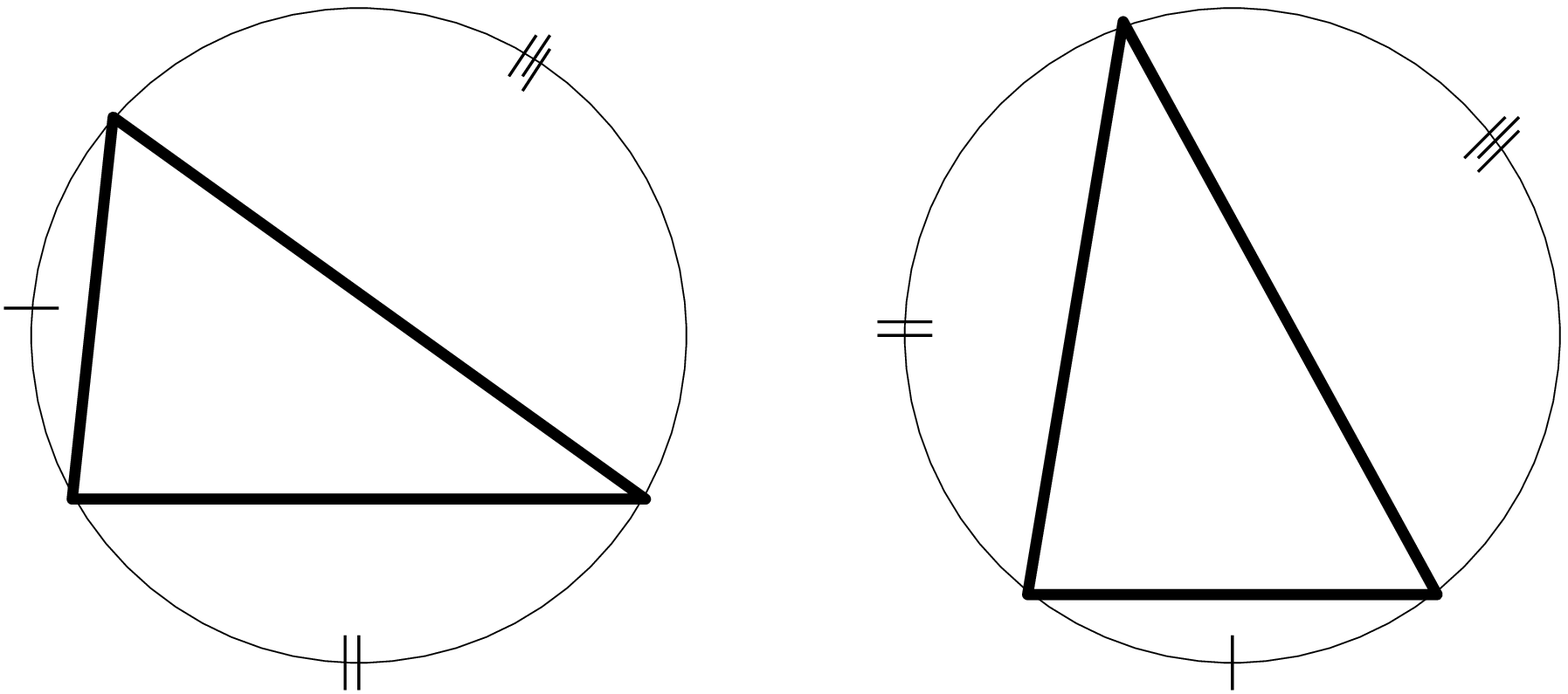}}

In each diagram the circle is divided into three arcs; mark the
shortest arc with a slash, the second-shortest arc with a double slash,
and the longest arc with a triple slash.  Note that $\Delta(\setA)$ consists of
the three arclengths from the left-hand circle, together with all sums
of two such arclengths, and three copies of $0$.  Thus the arclengths with a
slash and double-slash
in the left circle must equal the corresponding arclengths in the right
circle.  Hence the two triangles have equal angles, so they are congruent.
We can make their angles occur in the same order, by replacing $\setA$ by
$-\setA$ if necessary.  Then there is a rotation which makes the triangles
coincide.  Thus $\setA$ and $\setB$ are affinely equivalent.
\end{proof}

\begin{remark}
This proof shows that the multiplicative coefficient
can be taken to be $\pm 1$.
\end{remark}

We only need information about $k=3$ for our main result,
but we will say more about larger values of $k$ later in the
paper.
For example, $\SDA(n,4)$ is not always true, one counterexample
being the sets $\setA:=\{0,1,4,7\}_8$ and $\setB:=\{0,1,3,4\}_8$:
one easily checks that $\Delta(\setA)=\Delta(\setB)$, but $\setA$
and $\setB$ are affinely inequivalent since $\setA$ contains
arithmetic progressions of length $3$ but $\setB$ does not.
However, in Section~\ref{sec 4} we will prove $\SDA(n,4)$ for odd $n$,
and $\SDA(n,5)$ for $n$ coprime to $10$.  But in Section~\ref{sec 6} we
will give counterexamples to $\SDA(n,k)$ whenever $k>5$ and $n>2k+10$.


\section{Vanishing sums of roots of unity}
\label{sec roots}

The proof of our main result relies on some facts about vanishing
sums of roots of unity, which we discuss in this section.
Let $\settA$ be a multiset of roots of unity such that $\sum_{\alpha\in \settA}
\alpha=0$.
If the only sub-multisets of $\settA$ with zero sum are the empty set and $\settA$,
we say $\settA$ is a \emph{minimal} vanishing sum of roots of unity.  The
\emph{weight} of
the sum is the size of the multiset $\settA$.  We can multiply
any vanishing sum of roots of unity by an arbitrary root of unity, without
affecting minimality.

Vanishing sums of roots of unity have been studied extensively.
In our situation, it turns out that we need to understand
vanishing sums of twelve roots of unity.  In fact, it is shown in
\cite{poonen_rubin} that, up to multiplying by an arbitrary root of unity,
there are precisely $107$ minimal vanishing sums of weight at most $12$.
However, complications arise in passing from minimal vanishing sums to
nonminimal vanishing sums, since each minimal sum can be multiplied by
an arbitrary root of unity: thus there are infinitely many vanishing
sums of twelve roots of unity, so one cannot simply test them all.
We give a self-contained approach which does not require the results
from \cite{poonen_rubin}.

\begin{lemma}
\label{mann}
Let $\settA$ be a weight-$d$ minimal vanishing sum of roots of unity, and
suppose $1\in \settA$.  Let $n$ be the least common multiple of the orders
of the roots of unity in $\settA$.  Then $n$ divides the product of the primes
not exceeding $d$.  If $d$ is prime and $d\mid n$ then $\settA$ consists of
all the $d\tth$ roots of unity (so $n=d$).  If $d-1$ is prime and
$(d-1)\mid n$ then $n\mid 6(d-1)$.
\end{lemma}

\begin{proof}
First we show $n$ is squarefree.  If not then $n=rp^\ell$ where $p$ is prime,
$\ell>1$, and $r$ is coprime to $p$.  Let $\zeta$ be a primitive $n\tth$ root
of unity.  For any $i$ with $0\le i<n$, we can write $i=pa+b$ where $0\le b<p$,
so $\zeta^i=\zeta^b(\zeta^p)^a$.  Rewriting our sum of roots of unity in this
manner, we get a vanishing linear combination of $\zeta^0, \zeta^1,\dots,
\zeta^{p-1}$ with coefficients in $\Q(\zeta^p)$.  The field extension
$\Q(\zeta)/\Q(\zeta^p)$ has degree $[\Q(\zeta):\Q]/[\Q(\zeta^p):\Q]=
\phi(n)/\phi(n/p)=p$.  Hence our vanishing linear combination must have all
coefficients being zero.  By minimality, only one coefficient can be a
nontrivial sum of roots of unity.  Since the sum includes $1$, it follows that
every root of unity in the sum has order dividing $n/p$, a contradiction.  Thus
$n$ is squarefree.

Now let $p$ be a prime dividing $n$, and rewrite the sum as
$\sum_{\zeta^p=1}\zeta s_\zeta$ where each $s_\zeta$ is a sum of $(n/p)\tth$
roots of unity.  Since the sum includes $1$, the term $s_1$ is a nontrivial sum
of roots of unity.  By the definition of $n$, some other $s_\zeta$ must also be
a nontrivial sum of roots of unity.  By minimality, any $s_\zeta$ which is
nontrivial must be nonzero.  Letting $\mu_j$ be the set of $j\tth$ roots
of unity, we know that
$\Q(\mu_p,\mu_{n/p})=\Q(\mu_n)$ is an extension of $\Q(\mu_{n/p})$ of degree
$[\Q(\mu_n):\Q]/[\Q(\mu_{n/p}):\Q]=\phi(n)/\phi(n/p)=p-1$.  Thusthe polynomial
$x^{p-1}+x^{p-2}+\dots+1$ (whose roots are the primitive $p\tth$ roots of
unity) is
irreducible over $\Q(\mu_{n/p})$, so every $s_\zeta$ takes the same value.
In particular, each $s_\zeta$ is nonzero, so our sum has weight at least $p$,
whence
$n$ divides the product of the primes not exceeding $d$.  If $d=p$ then every
$s_\zeta$ is a single root of unity, and $s_1=1$, so every $s_\zeta=1$ and thus
our sum consists
of all the $p\tth$ roots of unity (and $n=p$).  Finally, suppose $d-1=p$.
Then all but one $s_\zeta$ consists of a single root of unity $\alpha$, and one
$s_\zeta$ is the sum of two roots of unity $\beta+\gamma$.  Since all $s_\zeta$
have the same value, we have
$-\alpha+\beta+\gamma=0$, which is a weight-three vanishing sum of roots of
unity,
and thus (from what we proved so far) must be a scalar times the sum of the
cube
roots of unity.  Hence both $\beta$ and $\gamma$ are sixth roots of unity times
$\alpha$.  Since our original sum includes $1$, one of $\alpha,\beta,\gamma$
equals $1$, so they all are sixth roots of unity and thus $n$ divides
$6p$.  This concludes the proof.
\end{proof}

\begin{remark}
All but the last sentence of the lemma was proved by Mann \cite{mann}.
\end{remark}

\begin{corollary}
\label{small relations}
Every minimal vanishing sum of roots of unity of weight $d<6$ has $d$
being prime and moreover has the form $\alpha\zeta_1+\alpha\zeta_2+\dots+
\alpha\zeta_d$
where $\alpha$ is a fixed root of unity and the $\zeta_i$ are all the distinct
$d\tth$ roots of unity.
\end{corollary}

\begin{proof}
Lemma~\ref{mann} proves this unless the sum is a scalar times a sum involving
only sixth roots of unity.  So consider a minimal vanishing sum of sixth
roots of unity, which we may assume includes $1$.  Rewrite this sum in the form
$\sum_{\zeta^3=1}\zeta s_\zeta$,
where each $s_\zeta$ is a sum of $1$'s and $-1$'s.  Since $1+x+x^2$ is
irreducible over $\Q$, every $s_\zeta$ must have the same value.  If this
common value is zero, then by minimality
our vanishing sum is $-1+1$.
If the common value is not zero, then by minimality our vanishing sum is
the sum of the cube roots of unity.
\end{proof}


\section{Proof of main result}
\label{sec 3}

In this section we complete the proof of Theorem~\ref{thm intro} by proving

\begin{theorem}
\label{k=3}
Let $\settA =\{\alpha_1,\alpha_2,\alpha_3,\bar\alpha_1,\bar\alpha_2,
\bar\alpha_3\}$ and $\settB =\{\beta_1,\beta_2,\beta_3,$
$\bar\beta_1,\bar\beta_2,\bar\beta_3\}$ be
multisets of $n^{\operatorname{th}}$ roots of unity, where
$\prod\alpha_j=1=\prod\beta_j$.
Suppose that the two multisets $\{\sum_{\phi\in \settA}\phi^r:1\le r\le n\}$ and
$\{\sum_{\psi\in \settB}\psi^r:1\le r\le n\}$ are identical.  Then there is
an integer $u$ coprime to $n$ such that $\settA =\{\psi^u:\psi \in \settB\}$.
\end{theorem}

First we show that this result implies Theorem~\ref{thm intro}.
\begin{proof}[Proof of Theorem~\ref{thm intro}]
The implication $(1)\Rightarrow(2)$ is Proposition~\ref{prop affpq},
the implication $(2)\Rightarrow(3)$ is Proposition~\ref{prop pqspec},
and the implication $(3)\Rightarrow(4)$ is obvious.  Thus it suffices
to prove $(4)\Rightarrow(1)$.  So let $A$ and $B$ be \zo 
circulants with residue sets $\setA:=\{a_1,\dots,a_k\}_n$ and
$\setB:=\{b_1,\dots,b_k\}_n$, where $k\le 3$, and suppose that $AA^T$ and $BB^T$
are similar.  To complete the proof, we must show that $\setA$ and $\setB$
are affinely equivalent.  For $k\le 2$ this was proved in Theorem~\ref{thm 2},
so suppose $k=3$.  By Proposition~\ref{prop SDA3}, it suffices to prove that
$\hat\setA:=\Delta(\setA)$ and $\hat\setB:=\Delta(\setB)$ are linearly
equivalent.  On the other hand, since $AA^T$ and $BB^T$ are similar, they
have the same eigenvalues; recalling their eigenvalues from (\ref{eigen}),
it follows that
\[
\left\{\sum_{1\le i,j\le 3} \zeta^{(a_i-a_j)r}:0\le r<n\right\}
=
\left\{\sum_{1\le i,j\le 3} \zeta^{(b_i-b_j)r}:0\le r<n\right\},
\]
where $\zeta$ is a fixed primitive $n\tth$ root of unity.
Write $\alpha_i:=\zeta^{a_i-a_{i+1}}$ for $1\le i\le 3$, where arithmetic
on indices is done modulo~$3$.  Then the $\alpha_i$ are $n\tth$ roots of unity
with $\prod_{i=1}^3\alpha_i=1$.  Put $\settA:=\{\alpha_1,\alpha_2,\alpha_3,
\bar\alpha_1,\bar\alpha_2,\bar\alpha_3\}$.  Define $\beta_i$ and $\settB$
similarly.  Then the equality of eigenvalues implies that
\[\left\{\sum_{\phi\in\settA}\phi^r:1\le r\le n\right\} =
\left\{\sum_{\psi\in\settB}\psi^r:1\le r\le n\right\}.\]
Now Theorem~\ref{k=3} implies there is an integer $u$ coprime to $n$ such that
$\settA=\{\psi^u:\psi\in\settB\}$.  Since $\settA\cup\{1,1,1\}=\{\zeta^a:a\in
\hat\setA\}$, it follows that $\hat\setA=u\hat\setB$, which as noted above
is sufficient to complete the proof.
\end{proof}

Our proof of Theorem~\ref{k=3} uses the following lemmas.

\begin{lemma}
\label{lcm trick}
Suppose $\settA$ and $\settB$ are multisets of $n\tth$ roots of unity with
$\#\settA=\#\settB$, and
the multisets $\mathcal{E}:=\{\sum_{\phi\in \settA}\phi^r:1\le r\le n\}$ and
$\{\sum_{\psi\in \settB}\psi^r:1\le r\le n\}$ are the same.  Then the
least common multiple $m_{\settA}$ of the orders of the elements of $\settA$
equals the corresponding
$m_{\settB}$.  Moreover, $\mathcal{E}$ consists of $n/m_{\settA}$ copies of
$\{\sum_{\phi\in \settA}\phi^r:1\le r\le m_{\settA} \}$.
\end{lemma}
\begin{proof}
This is similar to Propositon \ref{lcm_prop}.
The number $\sum_{\phi \in \settA}\phi^r$ equals $2(\#\settA)$ precisely when $r$
is divisible
by the stated least common multiple $m_\settA$, so the number of such $r$ in
$\{1,2,\dots,n\}$ equals $n/m_\settA$.  Since $\#\settA=\#\settB$, it follows that
$n/m_\settA=n/m_\settB$ and thus $m_\settA=m_\settB$.  The final assertion is
obvious.
\end{proof}

\begin{lemma}
\label{conjugation}
Suppose $\{\alpha_1,\dots,\alpha_k\}$ is a set of roots of unity
which includes two complex conjugate roots of unity.  Let $m$ be the least
common multiple of the orders of the various $\alpha_i/\alpha_j$.  Then every
$\alpha_i$ has order dividing $2m$.
\end{lemma}

\begin{proof}
If $\alpha_i=\bar\alpha_k$ then $\alpha_i^2=\alpha_i\bar\alpha_k=
\alpha_i/\alpha_k$
is an $m\tth$ root of unity, so $\alpha_i$ is a $(2m)\tth$ root of unity.
Since
every $\alpha_i/\alpha_j$ is an $m\tth$ root of unity, it follows that every
$\alpha_j$ has order dividing $2m$.
\end{proof}

We now prove our main result.

\begin{proof}[Proof of Theorem \ref{k=3}]
Suppose the multisets $\settA$ and $\settB$ provide a counterexample.
By Lemma~\ref{lcm trick}, we may assume that $n$ is the least common multiple
of the orders of the elements of $\settA$, and also that $n$ is the
corresponding least common multiple for $\settB$.  By hypothesis, the sum
of the elements of $\settA$
equals the sum of the $r\tth$ powers of the elements of $\settB$, for some $r$.
Thus $\sum_{a\in A} a + \sum_{b\in B} (-b^r) = 0$ is a vanishing sum of twelve
roots of unity.  For notational convenience, we will write this vanishing
sum as $\sum_{j=1}^3 (\gamma_j+\bar\gamma_j-\delta_j-\bar\delta_j)$, where
$\prod\gamma_j=1=\prod\delta_j$ and where moreover all $\gamma_j$'s and
$\delta_j$'s are $n\tth$ roots of unity and $n$ is the least common multiple
of the orders of either the $\gamma_j$'s or the $\delta_j$'s.
Note that multiplying the sum by $-1$ has the affect of switching the
$\gamma_j$'s and $\delta_j$'s.  In what follows, we implicitly use this
symmetry, as well as possibly relabeling the $\gamma_j$'s and $\delta_j$'s
or replacing all the $\gamma_j$'s (or $\delta_j$'s) by their complex
conjugates.

First we treat some small values of $n$, namely the values
$n\in\{ 840,$ $132, 90\}$ and their divisors.  A simple
{\sc MAGMA} program verifies the result in these cases.

Henceforth assume $n$ does not divide $840$, and consider a minimal
vanishing subsum
which includes a root of unity whose order does not divide $420$.
First suppose this subsum includes two complex conjugate roots of unity.
Let $m$ be the least common multiple of the orders of
the ratios of roots of unity involved in the subsum.
By Lemma~\ref{conjugation},
every root of unity in the subsum has order dividing $2m$.
By Lemma~\ref{mann},
$m$ divides either $2\cdot 3\cdot 5\cdot 7=210$ or $2\cdot 3\cdot 11=66$, so
we must have $m\mid 66$ and $11\mid m$.  Since $11\mid m$, Lemma~\ref{mann}
implies that the subsum has weight twelve, and all twelve roots of unity have
order dividing $132$, so $n\mid 132$, a case which was treated by our MAGMA
program.

Thus any root of unity in our vanishing sum whose order
does not divide $420$ must be contained in a minimal vanishing subsum which
does not include two complex conjugates, and hence has weight at most six.
Consider one such subsum.

If the weight is six
then the sum includes one element from each pair $(\gamma_j,\bar\gamma_j)$ and
one from each pair $(-\delta_j,-\bar\delta_j)$.
By Lemma~\ref{mann}, every element of the sum is a $30\tth$
root of unity times some fixed constant $c$.  Since $\prod\gamma_j=1$, we
see that
either $c$ or $c^3$ is a $30\tth$ root of unity.  Thus $n\mid 90$, a case
which was treated by our {\sc MAGMA} program.

If the weight is five
then we may assume the sum includes one element from each pair
$(\gamma_j,\bar\gamma_j)$
and one from $(-\delta_1,-\bar\delta_1)$ and one from
$(-\delta_2,-\bar\delta_2)$.
By Corollary \ref{small relations}, the roots of unity in this sum are fifth
roots
of unity times one another.  Since $\prod\gamma_j=1$, we see as above that the
$\gamma_j$ are $15\tth$ roots of unity, so $\delta_1$ and $\delta_2$ are
$30\tth$
roots of unity.  Thus $\delta_3=1/(\delta_1\delta_2)$ has order dividing $30$.
But $\delta_3+\bar\delta_3=0$ implies $\delta_3=-\bar\delta_3=-1/\delta_3$, so
$\delta_3^2=-1$, a contradiction.

There are no minimal vanishing sums of weight four, by
Corollary~\ref{small relations}.

Suppose the weight is three.  By Corollary \ref{small relations}, the three
roots of unity
in the sum are cube roots of unity times one another (and all three are
distinct).
Since the sum involves an element of order not dividing $420$, it follows that
all three roots of unity in the sum have order not dividing $420$.
If the sum includes only elements of the form $\gamma_j^{\pm 1}$, then since
$\prod\gamma_j=1$ we see that the $\gamma_j$'s are cube roots of unity, a
contradiction.
Thus, without loss we may assume the sum is
$\gamma_1+\gamma_2^{\pm 1}-\delta_1$.
Since the order of $-\delta_1$ does not divide $420$, and $\delta_2\delta_3=
1/\delta_1$,
we may assume the order of $\delta_2$ does not divide $420$.  Thus the minimal
vanishing subsum involving $-\delta_2$ does not include any complex conjugates,
and has weight at most three.  The possibilities are
\begin{enumerate}
\item[$(1)$] $\gamma_3^{\pm 1}-\delta_2-\delta_3^{\pm 1}=0$
\item[$(2)$] $-\delta_2-\delta_3^{\pm 1}=0$
\item[$(3)$] $\gamma_3^{\pm 1}-\delta_2=0$.
\end{enumerate}
In case (2) we have $\gamma_3+\bar\gamma_3=0$, so $\gamma_3=\pm \imag$ has
order $4$.
In cases (1) and (3), $\gamma_3^{\pm 1}$ is a sixth root of unity times
$\delta_2$,
so the order of $\gamma_3$ does not divide $420$.  Thus in every case
$\gamma_3$ is
not a cube root of unity, so $\bar\gamma_2$ is not a cube root of unity times
$\gamma_1$.

Hence the minimal vanishing subsum involving $\gamma_1$ is
$\gamma_1+\gamma_2-\delta_1$, so $\gamma_2=\omega\gamma_1$ and
$\delta_1=-\omega^2\gamma_1$ where $\omega$ is a primitive
cube root of unity.  Thus $\gamma_3=\omega^2/\gamma_1^2$.

In case (2) we have $\delta_2=-\delta_3^{\pm 1}$ and $\gamma_3=\pm \imag$, so
$\gamma_1$
and $\gamma_2$ have order $24$ while $\delta_1$ has order $8$, so we must have
$\delta_2=-\delta_3$ and thus $\delta_2$ and $\delta_3$ have order $16$.
But then $\#\langle\gamma_1,\gamma_2,\gamma_3\rangle=24$ and
$\#\langle\delta_1,\delta_2,\delta_3\rangle=16$,
which is a contradiction since neither of $16$ or $24$ divides the other.
In case (3) we have $\delta_2=\gamma_3^{\pm 1}$ and $\delta_3=\pm \imag$,
so $\delta_2=1/(\delta_1\delta_3)=\pm \imag\omega/\gamma_1$; since
$\gamma_3=\omega^2/\gamma_1^2$,
it follows that the order of $\gamma_1$ divides $12$, a contradiction.
So suppose we are in case (1).  Since $\delta_1=1/(\delta_2\delta_3)$ is not a
cube root of unity, it follows that $\delta_2$ is not a cube root of unity
times
$\bar\delta_3$, so we must have $\gamma_3^{\pm 1}-\delta_2-\delta_3=0$.  By
Corollary \ref{small relations},
$\delta_2\delta_3=\gamma_3^{\pm 2}=(\omega/\gamma_1^4)^{\pm 1}$, but also
$\delta_2\delta_3=1/\delta_1=-\omega/\gamma_1$, so the order of $\gamma_1$
divides $30$, a contradiction.

We have shown that every summand whose order does not divide $420$ is involved
in a minimal
vanishing sum of weight two.  Suppose $\{\gamma_j,\bar\gamma_j:1\le j\le 3\}=
\{\delta_j,\bar\delta_j:1\le j\le 3\}$.  From the definition of the $\gamma_j$
and $\delta_j$, it follows that $\{\alpha_j,\bar\alpha_j:1\le j\le 3\}=
\{\beta_j^r,\bar\beta_j^r:1\le j\le 3\}$.  In particular, the least common
multiple of the orders of the $\beta_j^r$ equals the corresponding least common
multiple for the $\alpha_j$, which we know equals the corresponding least
common multiple for the $\beta_j$, namely $n$.
Thus $r$ is coprime to $n$, contradicting our assumption that the $\alpha_j$
and $\beta_j$ are a counterexample to the desired result.

Next suppose that $\gamma_1=-\gamma_2$ and the order of $\gamma_1$ does not
divide $840$.
Then $\gamma_3=-1/\gamma_1^2$ has order not dividing $420$,
so we may assume $\gamma_3=\delta_3$.  Next, $\delta_1\delta_2=1/\delta_3$ has
order not
dividing $420$, so we may assume $\delta_1$ has order not dividing $420$,
whence
$\delta_1=-\delta_2^{\pm 1}$.  We do not have $\delta_1=-1/\delta_2$, since
that would imply $\delta_3=-1$.  Thus $\delta_1=-\delta_2$, so
$-1/\gamma_1^2=\gamma_3=\delta_3=-1/\delta_1^2$, whence $\gamma_1=\pm\delta_1$.
Thus $\{\gamma_1,\gamma_2,\gamma_3\}=\{\delta_1,\delta_2,\delta_3\}$,
and we have already achieved a contradiction in this situation.

Suppose that $\gamma_1=\delta_1$ and the order of $\gamma_1$ does not
divide $840$.
Since $\gamma_1\gamma_2\gamma_3=1$, we may assume the order of $\gamma_2$ does
not divide $840$ as well.  Thus, after possibly switching $\delta_2$ and
$\delta_3$,
we must have either $\gamma_2=\delta_2^{\pm 1}$ or $\gamma_2=-\bar\gamma_3$.
If $\gamma_2=\delta_2^{\pm 1}$, then
$\gamma_3+\bar\gamma_3-\delta_3-\bar\delta_3$
is a vanishing sum, but there are no minimal
vanishing sums of weights four or one so we must have
$\gamma_3=\delta_3^{\pm 1}$,
which is a case we have already handled.  Thus $\gamma_2=-\bar\gamma_3$, so
$\gamma_1=-1$, contradiction.

Finally, we may assume that the order of $\alpha_1$ does not divide $840$.
The minimal vanishing subsum involving $\alpha_1$ must have weight two, and is
not of the form
$\alpha_1+\alpha_j$ or $\alpha_1-\beta_j^{\pm r}$.  The only remaining
possibility is
$\alpha_1+\bar\alpha_j$; here we know $j\ne 1$, so we may assume $j=2$,
whence $\alpha_3=-1$.
Switching the roles of $\alpha_i$ and $\beta_i$,
we may also assume that $\beta_2=-\bar\beta_1$ and $\beta_3=-1$.
Since $\alpha_2=-\bar\alpha_1$,
at least one of $\alpha_1$ and $\alpha_2$ has even order, so by possibly
switching $\alpha_1$ and $\alpha_2$ we may assume that $\alpha_1$ has
even order. 
Then $n=\#\langle\alpha_1,\alpha_2,\alpha_3\rangle=\#\langle\alpha_1\rangle$.
Similarly we may assume that $\beta_1$ has even order, so $\beta_1$ has
order $n$.
Thus there is an $r$ coprime to $n$ such that $\alpha_1=\beta_1^r$, and it
follows that $\alpha_2=\beta_2^r$ and $\alpha_3=\beta_3^r$.  This again is a
contradiction, and as we have now treated every case it follows that the
supposed counterexample does not exist.
\end{proof}

In the above argument, we used that all the eigenvalues of $AA^T$ and $BB^T$
are the same.  One can actually obtain a lot of information from just the
hypothesis that these matrices have a single eigenvalue in common.  Namely,
by expanding on the above argument, one can show:

\begin{proposition}
If $\alpha_1,\alpha_2,\alpha_3,\beta_1,\beta_2,\beta_3$ are roots of unity
with $\prod \alpha_i=\prod \beta_i =1$ and $\sum (\alpha_i+\bar\alpha_i)=
\sum (\beta_i+\bar\beta_i)$, then the multisets
$\settA=\{\alpha_1,\alpha_2,\alpha_3,\bar\alpha_1,\bar\alpha_2,\bar\alpha_3\}$
and
$\settB=\{\beta_1,\beta_2,\beta_3,\bar\beta_1,\bar\beta_2,\bar\beta_3\}$ satisfy
one of the following, with
$\omega^3=1$,$\imag^4=1$,$\phi^5=1$,$\sigma^7=1$,$\mu^8=1$,$\nu^{16}=1$
being roots of unity with the orders indicated.
\begin{enumerate}
\item $\settA=\settB$.
\item $\settA=\{\alpha,-\bar\alpha,-1,\bar\alpha,-\alpha,-1\}$ and
$\settB=\{\beta,-\bar\beta,-1,\bar\beta,-\beta,-1\}$.
\item After possibly switching $\settA$ and $\settB$ we have one of the following
\begin{enumerate}
\item $\settA=\{\mu,-\mu,-\bar\mu^2,\bar\mu,-\bar\mu,-\mu^2\}$ and
$\settB=\{\omega,\omega^2,1,\omega^2,\omega,1\}$
\item $\settA=\{\phi,\phi^2,\bar\phi^3,\bar\phi,\bar\phi^2,\phi^3\}$ and \\
$\settB=\{\omega,-\omega\phi^2,-\omega\bar\phi^2,\omega^2,-\omega^2\bar\phi^2,
-\omega^2\phi^2\}$
\item $\settA=\{\omega,\sigma^3\omega,\bar\sigma^3\omega,\omega^2,
\bar\sigma^3\omega^2,\sigma^3\omega^2\}$ and \\
$\settB=\{-\sigma\omega,-\sigma\omega^2,\bar\sigma^2,
-\bar\sigma\omega^2,-\bar\sigma\omega,\sigma^2\}$
\item $\settA=\{\nu,-\nu,-\bar\nu^2,\bar\nu,-\bar\nu,-\nu^2\}$ and \\
$\settB=\{\omega\nu^2,\omega^2\nu^2,\bar\nu^4,\omega^2\bar\nu^2,
\omega\bar\nu^2,\nu^4\}$
\item $\settA=\{\omega,\omega\phi^2,\omega\bar\phi^2,\omega^2,
\omega^2\bar\phi^2,\omega^2\phi^2\}$ and \\
$\settB=\{\phi,\imag\phi^2,-\imag\bar\phi^3,\bar\phi,
-\imag\bar\phi^2,\imag\phi^3\}$
\item both $\settA$ and $\settB$ are among the multisets (with sum $-1$)
\begin{enumerate}
\item $\{\sigma,\sigma^2,\bar\sigma^3,\bar\sigma,\bar\sigma^2,\sigma^3\}$;
\item $\{\imag\omega,-\imag\omega,\omega,-\imag\omega^2,\imag\omega^2,
 \omega^2\}$;
\item $\{-\omega\phi,-\omega^2\phi,\bar\phi^2,-\omega^2\bar\phi,
-\omega\bar\phi,\phi^2\}$.
\end{enumerate}
\end{enumerate}
\end{enumerate}
\end{proposition}

\par\noindent
{\bf Note.} In the above proposition,
there is a solution which uses case 3.f.iii for both $\settA$ and $\settB$
but with $\phi$ a different primitive fifth root of unity in $\settB$
than $\settA$. This is the only case where we need different choices
of
$\omega, \imag, \phi, \sigma, \mu$ and $\nu$
in $\settA$ and $\settB$.


\section{The cases $k=4$ and $k=5$}
\label{sec 4}

New phenomena occur when we move to $k=4$.
For instance, in Section~\ref{sec equiv} we showed that
the residue sets $\{0,1,4,7\}_8$ and $\{0,1,3,4\}_8$ 
correspond to \zo circulants which are \pq\ equivalent
but not affinely equivalent.  Also, one can show that the
autocorrelation matrices of the circulants corresponding to
$\{0,1,2,6\}_{12}$ and $\{0,2,3,6\}_{12}$ are similar, but are
not conjugate via a permutation matrix (so the circulants are
not \pq\ equivalent).
These examples show that, when $k=4$, condition (2) of Theorem~\ref{thm intro}
does not imply condition (1), and condition (4) does not imply condition (3).

Still, for $k=4$ and $k=5$ we now show that Theorem~\ref{thm intro}
is true whenever $n$ is not divisible by small primes.
As in the case $k=3$, we proceed by proving $\SDA(n,k)$ and then
examining vanishing sums of roots of unity.

\begin{lemma}
\label{lem 4}
$\SDA(n,4)$ is true if $n$ is odd.
$\SDA(n,5)$ is true if $n$ is coprime to $10$.
\end{lemma}

\begin{proof}
Suppose $X:=\{x_1,x_2,x_3,x_4\}$ and $Y:=\{y_1,y_2,y_3,y_4\}$ are 4-element
subsets of $\Z/n\Z$ with $\Delta(X)=\Delta(Y)$.
Then there is a permutation $\pi$ of the set $T$ of order-$2$ subsets of
$\{1,2,3,4\}$, and a map $\sigma:T\to\{1,-1\}$, such that
\begin{equation}
\label{xyeq}
x_{\min(A)} - x_{\max(A)} = \sigma(A) (y_{\min(\pi(A))} - y_{\max(\pi(A))})\quad
\text{for every $A\in T$.}
\end{equation}

We tested via computer that, for every choice of $\pi$ and $\sigma$,
the above identity implies that $X$ and $Y$ are affinely equivalent
so long as $n$ is odd.  In fact, if we think of the $x_i$'s and $y_i$'s as
indeterminates, then there exist $\rho\in S_4$ and $c\in\{1,-1\}$ such that,
for each $1\le i\le 3$, the equation
\begin{equation}
\label{i4}
x_i-x_4 = c(y_{\rho(i)} - y_{\rho(4)})
\end{equation}
is a rational linear combination of the equations (\ref{xyeq})
(for any fixed choice of $\pi$ and $\sigma$).  Moreover, the denominators of
the coefficients in this combination have no prime factors besides $2$ and $13$.
This proves $\SDA(n,4)$ when $n$ is coprime to $26$.  Values of $n$ which are
odd multiples of $13$ require an additional argument: for such $n$, consider
a pair $(\pi,\sigma)$ for which the above linear combination has a coefficient
with denominator divisible by $13$.  It turns out that $13(x_i-x_j)$ and
$13(y_i-y_j)$ are $\Z$-linear combinations of the equations (\ref{xyeq}),
for every $1\le i,j\le 4$, so viewing $x_i,y_i$ as members of $\Z/nZ$, it follows
that all the $x_i$'s are congruent to one another mod~$n/13$, and likewise the
$y_i$'s.
We can translate the $x_i$'s and $y_i$'s so that (say) $x_1=y_1=0$, without
affecting the pairwise differences between $x_i$'s or $y_i$'s.  Thus we may assume
that every $x_i$ and $y_i$ is divisible by $n/13$, so it suffices to prove that
$\{x_i/(n/13)\}$ and $\{y_i/(n/13)\}$ are affinely equivalent subsets of $\Z/13\Z$.
In each case, it turns out that (\ref{xyeq}) allows one to write every $x_i$
and $y_i$ as a multiple of $y_4$, so we get two explicit subsets of $\Z/13\Z$ and
they do indeed turn out to be affinely equivalent.

We treated the case $k=5$ in a similar manner.  In this case it turns out 
that there is always a system of equations like (\ref{i4}) which are rational
linear combinations of the equations like (\ref{xyeq}), and the denominators
of the coefficients in these combinations are not divisible by any primes
besides $2$ and $5$.
\end{proof}

\begin{remark}
Experimentally, it seems that $\SDA(n,4)$ fails if and only if $n$ is
divisible by $8$, and that $\SDA(n,5)$ fails if and only if $n>8$ and
$\gcd(n,10)>1$.  However, such refinements of Lemma~\ref{lem 4} would
not affect our next result.
\end{remark}

\begin{theorem} If $k\in\{4,5\}$ and every prime factor of $n$ is greater
than $2k(k-1)$, then the properties $(1)$, $(2)$, $(3)$, $(4)$ from
Theorem~\ref{thm intro}
are equivalent conditions on $n$-by-$n$ \zo circulants $A$ and $B$ of weight~$k$.
\end{theorem}

\begin{proof}
Suppose $k$ and $n$ satisfy the hypotheses.
If the circulants corresponding to
$\setA:=\{a_1,\dots,a_k\}_n$ and $\setB:=\{b_1,\dots,b_k\}_n$
have similar autocorrelation matrices, then as in the previous
section by considering eigenvalues we find a vanishing sum $S$ of
$2k(k-1)$ $(2n)\tth$ roots of unity.  By Lemma~\ref{mann}, since the prime
factors of $n$ are larger than $2k(k-1)$, every minimal vanishing
subsum of $S$ must be a pair $(\alpha,-\alpha)$.
Since $n$ is odd, no two $n\tth$ roots of unity are negatives of
one another, so we must have an equality of multisets
\[ \{ \zeta^{a_i-a_j}:1\le i,j\le k \} =
 \{ \zeta^{(b_i-b_j)u}:1\le i,j\le k\} \]
for some integer $u$, where $\zeta$ is a primitive $n\tth$ root of unity.
By Lemma \ref{lcm trick}, we may assume $\gcd(u,n)=1$.
Thus, the multisets of differences $\Delta(\setA)$ and $\Delta(\setB)$
are linearly equivalent, so Lemma~\ref{lem 4} implies $\setA$ and $\setB$
are affinely equivalent.
\end{proof}


\section{The situation for $k>5$}
\label{sec 6}

When $k\ge 6$, there are examples showing that our spectral
approach will not work, even if the prime factors of $n$ are large.
More precisely, for every $k\ge 6$ and every $n>2k+10$,
we will exhibit two $n$-by-$n$ weight-$k$
\zo circulants which are not \pq\ equivalent but yet
have the same autocorrelation matrices.  We emphasize that this
shows one cannot relate \pq\ equivalence to affine equivalence by
means of autocorrelation similarity; but it may still be true
for $k\ge 6$ that \pq\ equivalence and affine equivalence are related
for some other reason.

For $k\ge 6$,
consider the sets of integers $\setA$ and $\setB$ defined as
the complements in $\{0,1,2,\dots,k+5\}$ of the sets
$\{1,2,4,6,k+1,k+2\}$ and $\{1,3,6,k+1,k+2,k+3\}$, respectively.
In other words,
\[\setA=\{0,3,5,k+3,k+4,k+5\}\cup\{7,8,9,\dots,k\}\] and
\[\setB=\{0,2,4,5,k+4,k+5\}\cup\{7,8,9,\dots,k\}.\]  We will show that
$\setA$ and $\setB$ have the same multisets of differences, i.e.,
$\Delta(\setA)=\Delta(\setB)$, but that for any $n>2k+10$ they
define $n$-by-$n$ \zo circulants which are not \pq\ equivalent.

\begin{proposition}
If $k\ge 6$ then $\Delta(\setA)=\Delta(\setB)$.
\end{proposition}

\begin{proof}
We compute $\Delta(\setA)\setminus\Delta(\setA\cap\setB)$.
Since $\setA\setminus(\setA\cap\setB)=\{3,k+3\}$, this consists
of all differences between two elements of $\setA$ which
involve $3$ or $k+3$.  Thus, this is
the multiset of elements $\pm i$ with $i$ in the union of the two multisets
$\{-3,0,2,k,k+1,k+2,4,5,6,\dots,k-3\}$ and
$\{-k-3,-k+2,0,1,2,4-k,5-k,6-k,\dots,-3\}$.
The corresponding multiset for $\setB$ is the union of
$\{-2,0,2,3,k+2,k+3,5,6,7,\dots,k-2\}$
and $\{-4,0,1,k,k+1,3,4,5,\dots,k-4\}$.
Thus $\Delta(\setA)\setminus\Delta(\setA\cap\setB)=
\Delta(\setB)\setminus\Delta(\setA\cap\setB)$, so $\Delta(\setA)=\Delta(\setB)$.
\end{proof}

Let $\setA_n$ and $\setB_n$ be the images of $\setA$ and $\setB$
in $\Z/n\Z$, and let $A_n$ and $B_n$ be the corresponding $n$-by-$n$
circulants.

\begin{proposition}
If $k\ge 6$ and $n>2k+10$, then $A_n$ and $B_n$ are not \pq\ equivalent.
\end{proposition}

\begin{proof}
First assume $k\ge 9$.  The top row of $A_n$ has dot-product $1$ with
precisely ten rows of $A_n$, namely the shifts of the top row by
$k+1,\dots,k+5$ in either direction.  The bound $n>2k+10$ ensures that
in computing these dot-products, it suffices to add or subtract an
integer from the indices: we need not reduce the indices mod~$n$.
Likewise, the top row of $A_n$ has dot-product $2$ with precisely two
rows of $A_n$, namely the shifts of the top row by $k-1$ in either
direction.  The componentwise product of the top row of $A_n$ with the sum of
these twelve rows is the vector
\[
{\bf a} := (4,0,0,2,0,1,0,0,\dots,0,1,0,0,0,1,3,2,0,0,\dots),\]
where the `4' occurs in position $0$ and the `3' occurs in position $k+4$.
The corresponding product for $B$ is
\[{\bf b} := (3,0,2,0,1,1,0,0,\dots,0,1,0,0,0,0,3,3,0,0,\dots),\]
where the entries `3' are in positions $0$, $k+4$ and $k+5$.
If $B_n$ were gotten by permuting the rows and columns of $A_n$, then
${\bf b}$ would be a permutation of ${\bf a}$.  But this is not the case,
since `4' is an entry in ${\bf a}$ but not in ${\bf b}$.

If $6\le k\le 8$ then the top row of $A_n$ has dot-product $2$ with
more than two rows of $A_n$, but its componentwise product with the sum of
all such rows is not a permutation of the corresponding product for $B_n$.
This proves the result in every case.
\end{proof}


\section{Relationship with the \Adam\ problem}
\label{sec adam}

\theoremstyle{definition}
\newtheorem*{AP}{\'Ad\'am Problem}
\newtheorem*{BAP}{Bipartite \'Ad\'am Problem}

A much-studied problem is the

\begin{AP}
For which $n$ do there exist $n$-by-$n$ \zo circulants
$A$ and $B$ which are not linearly equivalent but for which $B=PAP^{-1}$
for some permutation matrix $P$?
\end{AP}

Actually, \Adam\  made the conjecture that this \ppinv\ equivalence
was always the same as linear equivalence, but a counterexample was
published soon thereafter.
This led many authors to seek ways to weaken the original conjecture to
make it true.

This problem resembles the problem studied in this paper, which we now
rename:

\begin{BAP}
Describe the \zo circulants $A$ and $B$
which are not affinely equivalent but for which $B=PAQ$ for
some permutation matrices $P$ and $Q$.
\end{BAP}

\Adam\  was interested in isomorphisms between directed graphs that
had prescribed vertex transitive symmetry groups, in this case the
cyclic group.
Our problem asks the same question for directed bipartite graphs
which have the same group acting on each part.

Muzychuk has given a magnificent solution to the original
\Adam\ problem~\cite{muzychuk}:
the answer is all $n$ which are divisible by either $8$ or by the square of
an odd prime.  His proof uses detailed considerations of Schur algebras, among
other things.  As far as we know, the $n$'s occurring in solutions to the
bipartite \Adam\  problem might be precisely the $n$'s which solve the
\Adam\ problem; however, it seems that Muzychuk's method does not apply to
the bipartite problem.  We note, however, that Babai's group-theoretic
proof of the \Adam\ conjecture in the case of prime $n$ can be extended
(with some effort) to the bipartite situation.

There is evidence that these are different problems.
The first counterexample to the \Adam\  conjecture (from \cite{ET}) was the
pair of circulants $\{ 1, 2, 5 \}_8$ and $\{ 1, 5, 6 \}_8$.
On the other hand, Theorem~\ref{thm intro} shows that there
are no weight three ``counterexamples''
(permutation equivalent but not affinely equivalent)
for the bipartite \Adam\  problem.

However, there is often a connection between counterexamples
in the two problems.
Our original counterexample of circulants that are \pq\ but not affinely
equivalent was $\{ 0,1,4,7 \}_8$ and $\{ 0,1,3,4 \}_8$.
These two circulants do not directly form a counterexample for the
\Adam\ problem because they are not \ppinv\ equivalent.
But the affine equivalence class of $\{0,1,4,7\}_8$ includes
$\{0,1,2,5\}_8$, which is \ppinv\ equivalent to $\{0,1,5,6\}_8$, which in
turn is affinely equivalent to $\{0,1,3,4\}_8$.  Thus, by picking different
members of the two affine equivalence classes, we can turn our
bipartite \Adam\ counterexample into an \Adam\ counterexample.

We did some computer searches to find bipartite \Adam\ counterexamples,
and most of the examples we found were affinely equivalent to \Adam\
counterexamples.  However, in weight six there are bipartite
\Adam\  counterexamples like
$\{ 0 , 1, 2, 5, 8, 10 \}_{16}$ and $\{ 0, 2, 3, 7, 8, 10 \}_{16}$
which are not affinely equivalent to \Adam\  counterexamples.
We do not know how rare such examples will be for larger weights.


\section{Acknowledgements and Future Directions}
\label{sec ack}

It is a pleasure to thank Bradley Brock for helpful discussions
on vanishing sums of roots of unity, and Marshall W. Buck for
help with equation solving.

There are several different directions for future research depending on
what equivalence relations and what parameter ranges are of the most interest.
For larger densites, say $ k \approx n/2 $,
it may still be true that the autocorrelation spectra usually determine
the affine equivalence classes.

In some applications it is natural to study sparse circulants whose nonzero
entries can be $-1$ as well as $1$.
Also, the notion of spectral equivalence could be relaxed to simply
having the same mimimal polynomial or having the same characteristic
polynomial modulo some fixed prime.


\end{document}